\documentclass[11pt]{article}
\usepackage{amsfonts}
\usepackage{graphicx}
\usepackage{epsfig}
\usepackage{amssymb, amsmath}
\usepackage{amsthm}
\usepackage{setspace}
\usepackage[top = 1in, bottom = 1in, left = 1.2in, right =
1.2in]{geometry}

\newtheorem{theorem}{Theorem}[section]

\newtheorem{lemma}[theorem]{Lemma}

\newcommand {\eps}{\varepsilon}

\def \<{\langle}
\def \>{\rangle}

\def \eps{\varepsilon}

\begin{document}

\title{Acoustic Limit for the Boltzmann equation \\
in Optimal Scaling}

\singlespacing

\author{Yan Guo \\
\textit{Brown University} \\
\textit{guoy@cfm.brown.edu} \vspace*{.10in} \\
Juhi Jang\\
\textit{Courant Institute}\\
\textit{juhijang@cims.nyu.edu} \vspace*{.10in} \\
Ning Jiang\\
\textit{Courant Institute}\\
\textit{njiang@cims.nyu.edu}}
\maketitle

\begin{abstract}
Based on a recent $L^{2}\mbox{-}L^{\infty }$ framework, we establish
the acoustic limit of the Boltzmann equation for general collision
kernels. The scaling of the fluctuations with respect to Knudsen
number is optimal. Our approach is based on a new analysis of the
compressible Euler limit of the Boltzmann equation, as well as
refined estimates of Euler and acoustic solutions.
\end{abstract}



\numberwithin{equation}{section}


\section{Introduction and Main Results}

We study the Boltzmann equation
\begin{equation}
\partial _{t}F^{\varepsilon }+v\cdot \nabla _{x}F^{\varepsilon }=\frac{1}{%
\varepsilon }\mathcal{Q}(F^{\varepsilon },F^{\varepsilon })
\label{boltzmann}
\end{equation}%
where $F^{\varepsilon }(t,x,v)\geq 0$ is the density of particles of
velocity $v\in \mathbf{R}^{3},$ and position $x\in \Omega =$
$\mathbf{R}^{3}$ or $\mathbf{T}^{3},~$a periodic box. The positive
parameter $\varepsilon$ is the Knudsen number. Throughout this
paper, the collision operator takes the form
\begin{equation}
\begin{aligned} {\mathcal{Q}}(F_1\,,
F_2)(v)&=\int_{\mathbf{R}^3}\int_{\mathbf{S}^2}|v-u|^\gamma
F_1(u')F_2(v')B(\theta)\, d\omega\, du\\
&-\int_{\mathbf{R}^3}\int_{\mathbf{S}^2}|v-u|^\gamma F_1(u)F_2(v)B(\theta)\,
d\omega\, du\,, \end{aligned}
\end{equation}
where $-3 < \gamma \leq 1$, $u^{\prime}=u+[(v-u)\!\cdot\!\omega]\omega$, $%
v^{\prime}=v-[(v-u)\!\cdot\!\omega]\omega$, $\cos
\theta=(u-v)\!\cdot\!\omega/|v-u|$, and $0 < B(\theta) \leq C|\cos(\theta)|$%
. Such collision operators cover both the hard-sphere interaction and
inverse power law with an angular cutoff. The hard potential means $0 \leq
\gamma \leq 1$, while soft potential means $-3 < \gamma <0$.

\subsection{Hilbert Expansion}

We define a family of special distribution functions $\mu (t,x,v)$ the \emph{%
local Maxwellians} by
\begin{equation}
\mu (t\,,x\,,v)\equiv \frac{\rho (t,x)}{[2\pi T(t,x)]^{3/2}}\exp \left\{ -%
\frac{[v-\mathfrak{u}(t,x)]^{2}}{2T(t,x)}\right\}   \label{localmax}
\end{equation}%
which are equilibrium of the collision process:
\begin{equation*}
\mathcal{Q}(\mu \,,\mu )=0\,.
\end{equation*}%
$(\rho \,,\mathfrak{u}\,,T)$ represent the macroscopic density, bulk
velocity, and temperature respectively. If $(\rho \,,\mathfrak{u}\,,T)$ are
constant in $t$ and $x,$ $\mu $ is called a \emph{global Maxwellian}. It was
shown in \cite{Caf,GJJ} that for hard-sphere interaction, namely $\gamma =1$%
, as $\varepsilon \rightarrow 0$, $\{F^{\varepsilon }\}$ solutions to the
Boltzmann equation \eqref{boltzmann} converge to a local Maxwellian $\mu $
induced by a solution to the compressible Euler system:
\begin{equation}
\begin{split}
\partial _{t}\rho +\nabla _{x}\!\cdot \!({\rho }\mathfrak{u})& =0 \\
\partial _{t}(\rho \mathfrak{u})+\nabla _{x}\!\cdot \!(\rho \mathfrak{u}%
\otimes \mathfrak{u})+\nabla _{x}p& =0 \\
\partial _{t}\left[ \rho (e+\tfrac{1}{2}|\mathfrak{u}|^{2})\right] +\nabla
_{x}\!\cdot \!\left[ \rho \mathfrak{u}(e+\tfrac{1}{2}|\mathfrak{u}|^{2})%
\right] +\nabla _{x}\!\cdot \!(p\mathfrak{u})& =0
\end{split}
\label{Euler}
\end{equation}%
with the equation of state
\begin{equation}
p=\rho RT=\tfrac{2}{3}\rho e  \label{ES}
\end{equation}%
as long as the solution stays smooth. Let $(\rho (t,x),\mathfrak{u}%
(t,x),T(t,x))$ be a smooth solution of the Euler equations
(\ref{Euler}) for $t\in \lbrack 0,\tau ],$ $x\in \Omega $. Consider
the local Maxwellian $\mu $ from $(\rho ,\mathfrak{u},T)$ as in
\eqref{localmax}. As in \cite{Caf}, we take the Hilbert expansion of
solutions around $F_{0}\equiv \mu $ with the form
\begin{equation*}
F^{\varepsilon }=\sum_{n=0}^{5}\varepsilon ^{n}F_{n}+\varepsilon
^{3}F_{R}^{\varepsilon }\,,
\end{equation*}%
where $F_{0},...,F_{5}$ are the first 6 terms of the Hilbert expansion,
independent of $\varepsilon ,$ which solve the equations
\begin{equation*}
\begin{split}
0& =\mathcal{Q}(F_{0},F_{0}), \\
\{\partial _{t}+v\cdot \nabla _{x}\}F_{0}& =\mathcal{Q}(F_{0},F_{1})+%
\mathcal{Q}(F_{1},F_{0}), \\
\{\partial _{t}+v\cdot \nabla _{x}\}F_{1}& =\mathcal{Q}(F_{0},F_{2})+%
\mathcal{Q}(F_{2},F_{0})+\mathcal{Q}(F_{1},F_{1}), \\
& \;\;\;\;... \\
\{\partial _{t}+v\cdot \nabla _{x}\}F_{5}& =\mathcal{Q}(F_{0},F_{6})+%
\mathcal{Q}(F_{6},F_{0})+\sum_{\substack{ i+j=6 \\ 1\leq i\leq 5,1\leq j\leq
5}}\mathcal{Q}(F_{i},F_{j}).
\end{split}%
\end{equation*}%
We can construct smooth $F_{1}(t,x,v),F_{2}(t,x,v),...,F_{6}(t,x,v)$ for $%
0\leq t\leq \tau .$ For more detailed discussion, see \cite{Caf}. Now we put
$F^{\varepsilon }=\sum_{n=0}^{5}\varepsilon ^{n}F_{n}+\varepsilon
^{3}F_{R}^{\varepsilon }$ into the Boltzmann equation \eqref{boltzmann} to
derive the remainder equation for $F_{R}^{\varepsilon }$
\begin{equation}
\begin{split}
& \partial _{t}F_{R}^{\varepsilon }+v\cdot \nabla _{x}F_{R}^{\varepsilon }-%
\frac{1}{\varepsilon }\{\mathcal{Q}(\mu ,F_{R}^{\varepsilon })+\mathcal{Q}%
(F_{R}^{\varepsilon },\mu )\} \\
=\,& \varepsilon ^{2}\mathcal{Q}(F_{R}^{\varepsilon },F_{R}^{\varepsilon
})+\sum_{i=1}^{5}\varepsilon ^{i-1}\{\mathcal{Q}(F_{i},F_{R}^{\varepsilon })+%
\mathcal{Q}(F_{R}^{\varepsilon },F_{i})\}+\varepsilon ^{2}A
\end{split}
\label{remainder}
\end{equation}%
where
\begin{equation}
A=-\{\partial _{t}+v\cdot \nabla _{x}\}F_{5}+\sum_{i+j\geq 6,1\leq i,j\leq
5}\varepsilon ^{i+j-6}\mathcal{Q}(F_{i},F_{j}).
\end{equation}

The acoustic system is the linearization about the homogeneous state of the
compressible Euler system. After a suitable choice of units, the fluid
fluctuations $(\sigma ,u,\theta )$ satisfy
\begin{equation}
\begin{aligned} \partial_t \sigma+\nabla_{\!\!x}\!\cdot\! u=0\,,\qquad
&\sigma(x,0)=\sigma^{0}(x)\,,\\ \partial_t
u+\nabla_{\!x}(\sigma+\theta)=0\,,\qquad & u(x,0)=u^{0}(x)\,,\\
\tfrac{3}{2}\partial_t \theta+\nabla_{\!\!x}\!\cdot\! u=0\,,\qquad
&\theta(x,0)=\theta^{0}(x)\,. \end{aligned}  \label{acoustic-system}
\end{equation}%
Such acoustic system (\ref{acoustic-system}) can be formally derived
from the Boltzmann equation \eqref{boltzmann} by letting
\begin{equation}
F^{\varepsilon }=\mu ^{0}+\delta G^{\varepsilon }  \label{g}
\end{equation}%
where $\mu ^{0}$ is the global Maxwellian which corresponds to $\rho =T=1$
and $\mathfrak{u}=0$:
\begin{equation*}
\mu ^{0}\equiv \tfrac{1}{(2\pi )^{3/2}}\exp (-\tfrac{|v|^{2}}{2})
\end{equation*}%
and the fluctuation amplitude $\delta $ is a function of $\varepsilon $
satisfying
\begin{equation}
\delta \rightarrow 0\;\text{ as }\;\varepsilon \rightarrow 0\,.
\label{delta}
\end{equation}%
For instance, one can take
\begin{equation*}
\delta =\varepsilon ^{m}\text{ for any }m>0\,.
\end{equation*}%
With the above scalings, $G^{\varepsilon }$ formally converges to
\begin{equation}
G=\left\{ \sigma +v\cdot u+\left( \tfrac{|v|^{2}-3}{2}\right) \theta
\right\} \mu ^{0}  \label{gg}
\end{equation}%
as $\varepsilon \rightarrow 0$, where $\sigma ,u,\theta $ satisfy the
acoustic system \eqref{acoustic-system}. For detailed formal derivation, see
\cite{BGL1, GL}.

\subsection{Main Theorems}

The endeavor to understand how fluid dynamical equations for both
compressible and incompressible flows can be derived from kinetic
theory goes back to the founding work of Maxwell \cite{Max} and
Boltzmann \cite{Bol}. Most of these derivations are well understood
at several formal levels by now, and yet their full mathematical
justifications are still incomplete. In fact, the purpose of the
Hilbert's sixth problem \cite{Hilbert} is to seek a unified theory
of the gas dynamics including various levels of descriptions from a
mathematical standpoint. So far, there are basically three different
approaches mathematically. The first is based on spectral analysis
of the semi-group generated by the linearized Boltzmann equation,
see \cite{BU, KMN, Nishida}. The second is based on  Hilbert or
Chapman-Enskog expansions \cite{Caf, DEL}, see more recent work in
\cite{g3,j3},  \cite{G1, G2}. The third approach was the program
initiated from \cite{BGL1, BGL2}, working in the framework of global
renormalized solutions after the celebrated work of DiPerna-Lions
\cite{DL}, to justify global weak solutions of incompressible flows
(Navier-Stokes, Stokes, and Euler), and (compressible) acoustic
system, see
\cite{BGL1,BGL2,BGL3,GL,GS,JLM,LM,LM1,LM2,Saint-Raymond}.

The authors in \cite{GL} proved the convergence of the acoustic limit from
DiPerna-Lions solutions of the Boltzmann equation \eqref{boltzmann} with the
restriction on the size of fluctuations: $m>\tfrac{1}{2}$. Recently, in \cite%
{JLM}, this restriction has been relaxed to the borderline case $m=\tfrac{1}{%
2}$ by employing some new nonlinear estimates developed in \cite{LM} and a
new $L^{1}$ averaging lemma in \cite{GS}. However, due to some technical
difficulties mainly caused by the lack of local conservation laws and
regularity of renormalized solutions, the case for $m<\frac{1}{2}$ remains
an open question. On the other hand, in the framework of classical
solutions, in \cite{jj}, the authors have established the global-in-time
uniform energy estimates and proven the strong convergence for $m=1$, by
adapting the nonlinear energy method of \cite{g3,j3}. Although this method
displays in a clear way how the dissipation disappears in the acoustic limit
in terms of instant energies and dissipation rates, it does not cover other
interesting cases $0<m<1$ due to weak dissipations.

The purpose of this article is to establish the acoustic limit for $0<m<1$
via a recent $L^{2}\mbox{-}L^{\infty }$ framework. We will use $\delta $
instead of $\varepsilon ^{m}$ to denote the fluctuation amplitude. Since our
interest is the case of $0<m<1$ towards the optimal scaling, throughout the
paper, we assume that in addition to \eqref{delta},
\begin{equation}
\frac{\varepsilon }{\delta }\rightarrow 0\;\text{ as }\;\varepsilon
\rightarrow 0\,.  \label{delta2}
\end{equation}

\begin{theorem}
\label{main-theorem} Let $\tau >0$ be any given finite time and let
\begin{equation}
\sigma (0,x)=\sigma ^{0}(x),\;u(0,x)=u^{0}(x),\;\theta (0,x)=\theta
^{0}(x)\;\in H^{s},\;s\geq 4  \label{initial-data}
\end{equation}%
be any given initial data to the acoustic system \eqref{acoustic-system}.
Then there exist an $\varepsilon _{0}>0$ and a $\delta _{0}>0$ such that for
each $0<\varepsilon \leq \varepsilon _{0}$ and $0<\delta \leq \delta _{0},$
there exists a constant $C>0$ so that
\begin{equation}
\sup_{0\leq t\leq \tau }\Vert G^{\varepsilon }(t)-G(t)\Vert _{\infty
}+\sup_{0\leq t\leq \tau }\Vert G^{\varepsilon }(t)-G(t)\Vert _{2}\leq C\{%
\frac{\varepsilon }{\delta }+\delta \}  \label{AL}
\end{equation}%
where $\frac{\varepsilon }{\delta }\rightarrow 0\;$ as $\;\varepsilon
\rightarrow 0\,,$ and $G^{\varepsilon }$ and $G$ are defined in \eqref{g}
and \eqref{gg}, and $C$ depends only on $\tau $ and the initial data $\sigma
^{0},u^{0},\theta ^{0}$.
\end{theorem}

Our proof is different from the previous approach. Instead of estimating $%
G^{\varepsilon }-G$ directly, we make a detour to control of $G^{\varepsilon
}-G$ in two steps. The first step (Section 2) is to show as $\varepsilon
\rightarrow 0$,  $F^{\varepsilon }$ is close to the local Maxwellian $\mu
^{\delta },$ constructed from the smooth solution of the compressible Euler
equation. In fact, we are able to establish (Theorem \ref{Main-Theorem})
\begin{equation*}
F^{\varepsilon }-\mu ^{\delta }=O(\varepsilon ),
\end{equation*}%
before the time of possible shock formation, which is of the order of $\frac{%
1}{\delta }$ in the acoustic scaling (longer than any fixed time $\tau $!)$.$
The second step (Section 3) is to show that (Lemma \ref{lemma4.1}), within
the time scale of $\frac{1}{\delta },$\textbf{\ }%
\begin{equation*}
\mu ^{\delta }-\mu ^{0}=\delta G+O(\delta ^{2}).
\end{equation*}%
Such an estimate confirms that the solution of the acoustic equation $G$ is
the first order (linear) approximation of that to the Euler equations.
Combining these two estimates and comparing with \eqref{g}, we deduce our
theorem by dividing $\delta .$ Our proof relies on the existence of global
in-time smooth solutions to the linear acoustic system (\ref{acoustic-system}%
).

Our main technical contribution is a new analysis of the classical
compressible Euler limit to complete step one above. To precisely state our
result, we use the standard notation $H^{s}$ to denote the Sobolev space $%
W^{s,2}(\Omega )$ with corresponding norm $\Vert \cdot \Vert _{H^{s}}$. We
also use the standard notation $\Vert \cdot \Vert _{2}$ and $\Vert \cdot
\Vert _{\infty }$ to denote $L^{2}$ norm and $L^{\infty }$ norm in both $%
(x,v)\in \Omega \times \mathbf{R}^{3}$ variables. We use $\langle \cdot
\,,\cdot \rangle $ to denote the standard $L^{2}$ inner product. We also
define a weighted $L^{2}$ norm
\begin{equation*}
\Vert g\Vert _{\nu }^{2}=\int_{\Omega \times \mathbf{R}^{3}}g^{2}(x\,,v)\nu
(v)\,dx\,dv\,,
\end{equation*}%
where the collision frequency $\nu (v)\equiv \nu (\mu )(v)$ is defined as
\begin{equation*}
\nu (\mu )=\int_{\mathbf{R}^{3}}B(\theta )|v-v^{\prime }|^{\gamma }\mu
(v^{\prime })\,dv^{\prime }d\omega \,.
\end{equation*}%
Note that for given $-3<\gamma \leq 1$,
\begin{equation*}
\nu (\mu )\sim \rho \,(1+|v|)^{\gamma }.
\end{equation*}%
Define the linearized collision operator $\mathcal{L}$ by
\begin{equation*}
{\mathcal{L}}g=-\frac{1}{\sqrt{\mu }}\{\mathcal{Q}(\mu ,\sqrt{\mu }g)+%
\mathcal{Q}(\sqrt{\mu }g,\mu )\}.
\end{equation*}%
Let $\mathbf{P}g$ be the $L_{v}^{2}$ projection with respect to $[\sqrt{\mu }%
,v\sqrt{\mu },|v|^{2}\sqrt{\mu }].$ Then it is well-known that there exists
a positive number $c_{0}>0$ such that
\begin{equation}
\langle {\mathcal{L}}g,g\rangle \geq c_{0}\Vert \{{\mathbf{I-P}}\}g\Vert
_{\nu }^{2}\,.  \label{coercivity}
\end{equation}%
The solutions to the Boltzmann equation \eqref{boltzmann} are constructed
near the local Maxwellian of the compressible Euler system. So it is natural
to rewrite the remainder
\begin{equation}
F_{R}^{\varepsilon }=\sqrt{\mu }f^{\varepsilon }.  \label{f}
\end{equation}%
Since $\mu $ is a local Maxwellian, the equation of the remainder includes
the new term $\sqrt{\mu }^{-1}(\partial _{t}+v\!\cdot \!\nabla _{\!x})\sqrt{%
\mu }f^{\varepsilon }$. At large velocities, the distribution functions may
be growing rapidly due to streaming. To remedy this difficulty, following
Caflisch \cite{Caf}, we introduce a global Maxwellian
\begin{equation*}
\mu _{M}=\frac{1}{(2\pi T_{M})^{3/2}}\exp \left\{ -\frac{|v|^{2}}{2T_{M}}%
\right\} .
\end{equation*}%
where $T_{M}$ satisfies the following condition
\begin{equation}
T_{M}<\max\limits_{t\in \lbrack 0,\tau ],x\in \Omega }T(t,x)<2T_{M}\,.
\label{assumption-T}
\end{equation}
Note that under the assumption \eqref{assumption-T}, there exist constants $%
c_{1}\,,c_{2}$ such that for some $1/2<\alpha <1$ and for each $%
(t\,,x\,,v)\in \lbrack 0,\tau ]\times \Omega \times \mathbf{R}^{3}$, the
following holds
\begin{equation}
c_{1}\mu _{M}\leq \mu \leq c_{2}\mu _{M}^{\alpha }.  \label{bound}
\end{equation}%
We further define
\begin{equation}
F_{R}^{\varepsilon }=\{1+|v|^{2}\}^{-\beta }\sqrt{\mu _{M}}h^{\varepsilon
}\equiv \frac{1}{w(v)}\sqrt{\mu _{M}}h^{\varepsilon }  \label{h}
\end{equation}%
for any fixed
\begin{equation*}
\beta \geq \frac{9-2\gamma }{2}.
\end{equation*}

We now state the result on the compressible Euler limit:

\begin{theorem}
\label{Main-Theorem} Assume that the solution to the Euler equations $[\rho
(t,x),u(t,x),T(t,x)]$ is smooth and $\rho (t,x)$ has a positive lower bound
for $0\leq t\leq \tau $. Furthermore, assume that the temperature $T(t,x)$
satisfies the condition \eqref{assumption-T}. Let
\begin{equation}\label{initial}
F^{\varepsilon }(0,x,v)=\mu (0,x,v)+\sum_{n=1}^{5}\varepsilon
^{n}F_{n}(0,x,v)+\varepsilon ^{3}F_{R}^{\varepsilon }(0,x,v)\geq 0.
\end{equation}%
Then there is an $\varepsilon _{0}>0$ such that for $0<\varepsilon \leq
\varepsilon _{0},$ and for any $\beta \geq \frac{9-2\gamma }{2}$, there
exists a constant $C_{\tau }(\mu ,F_{0},F_{1},..F_{6})$ such that
\begin{equation}
\begin{split}
& \sup_{0\leq t\leq \tau }\varepsilon ^{\frac{3}{2}}\left\Vert \sqrt{\mu }%
^{-1}(1+|v|^{2})^{\beta }F_{R}^{\varepsilon }(t)\right\Vert _{\infty
}+\sup_{0\leq t\leq \tau }\left\Vert \sqrt{\mu }^{-1}F_{R}^{\varepsilon
}(t)\right\Vert _{2} \\
& \leq C_{\tau }\left\{ \varepsilon ^{\frac{3}{2}}\left\Vert \sqrt{\mu }%
^{-1}(1+|v|^{2})^{\beta }F_{R}^{\varepsilon }(0)\right\Vert _{\infty
}+\left\Vert \sqrt{\mu }^{-1}F_{R}^{\varepsilon }(0)\right\Vert
_{2}+1\right\} ,
\end{split}%
\end{equation}%
where $F_{R}^{\varepsilon }$ is the solution to the remainder equation %
\eqref{remainder}.
\end{theorem}

\noindent{\bf Remark:} Applying the bound \eqref{bound}, Lemma A.1
and Lemma A.2 of \cite{g3}, we can carefully choose 
$F_{R}^{\varepsilon }(0,x,v)$ in \eqref{initial} so that the initial
data \eqref{initial} are non-negative. Because the argument is quite
similar, we omit the details here. \\

Based on the a priori estimates given in Theorem \ref{Main-Theorem},
following the arguments in the pioneering work of Caflisch \cite{Caf}, we
can immediately derive the compressible Euler limit as well as the existence
of the solutions to the Boltzmann equation. As in \cite{Caf}, the Hilbert
expansion provides a natural way to establish an uniform in $\varepsilon $
control for the Euler limit. However, it was well-known \cite{Caf} that an $%
|v|^{3}f^{\varepsilon }$ term due to streaming in the $L^{2}$ estimate
creates an unpleasant analytical difficulty. We employ both $L^{2}$ and $%
L^{\infty }$ estimate with polynomial velocity weight \cite{G2,GJJ} to
control such a term with a high power of velocity\thinspace\ $v.$ On the one
hand, our analysis requires an additional assumption of moderate temperature
variation (\ref{assumption-T}). On the other hand, we do not need the
truncation of the Hilbert expansion as in \cite{Caf}, so that the positivity
of the solution is guaranteed. In particular, our theorem is designed to
apply to the acoustic limit because the temperature variation is only of the
order $\delta .$ Moreover, a cutoff trick used in \cite{SG} enables us to
treat all soft potentials $-3<\gamma \leq 1$ with an angular cutoff.

\section{Compressible Euler Limit}

In this section, we prove Theorem \ref{Main-Theorem}. Note that it suffices
to estimate $\|f^{\varepsilon }(t)\|_{2}$ and $\|h^{\varepsilon
}(t)\|_{\infty }$ to conclude the theorem. The proof relies on an interplay
between $L^{2}$ and $L^{\infty }$ estimates for the Boltzmann equation \cite%
{G2,GJJ}: $L^2$ norm of $f^\eps$ is controlled by the $L^\infty$ norm of the
high velocity part and vice versa. These uniform $L^2\mbox{-}L^\infty$
estimates are stated in the following two lemmas:

\begin{lemma}
\label{l2}\emph{($L^{2}\mbox{-}$Estimate):} Let $(\rho ,\mathfrak{u},T)$ be
a smooth solution to the Euler equations such that $\rho $ has a positive
lower bound and $T$ satisfies the condition \eqref{assumption-T}. Let $f^{%
\eps}\,,h^{\eps}$ be defined in \eqref{f} and \eqref{h}, and $c_{0}>0$ be as
in the coercivity estimate \eqref{coercivity}. Then there exists $%
\varepsilon _{0}>0$ and a positive constant $C=C(\mu \,,F_{0},F_{1}\,,\cdots
\,,F_{6})>0$, such that for all $\varepsilon <\varepsilon _{0}$
\begin{equation}
\frac{d}{dt}\Vert f^{\varepsilon }\Vert _{2}^{2}+\frac{c_{0}}{2\varepsilon }%
\Vert \{\mathbf{I}-\mathbf{P}\}f^{\varepsilon }\Vert _{\nu }^{2}\leq C\{%
\sqrt{\varepsilon }\Vert \varepsilon ^{3/2}h^{\varepsilon }\Vert _{\infty
}+1\}(\Vert f^{\varepsilon }\Vert _{2}^{2}+\Vert f^{\varepsilon }\Vert
_{2})\,.  \label{l2estimate}
\end{equation}
\end{lemma}

\begin{lemma}
\label{linfty}\emph{($L^{\infty }\mbox{-}$Estimate):} Let $(\rho ,\mathfrak{u%
},T)$ be a smooth solution to the Euler equations such that $\rho $ has a
positive lower bound and $T$ satisfies the condition \eqref{assumption-T}.
Let $f^{\eps}\,,h^{\eps}$ and $c_{0}>0$ be the same as in Lemma \ref{l2}.
Then there exist $\varepsilon _{0}>0$ and a positive constant $C=C(\mu
\,,c_{0},F_{1}\,,\cdots \,,F_{6})>0$, such that for all $\varepsilon
<\varepsilon _{0}$
\begin{equation}
\sup_{0\leq s\leq \tau }\{\varepsilon ^{3/2}\Vert h^{\varepsilon }(s)\Vert
_{\infty }\}\leq C\{\Vert \varepsilon ^{3/2}h_{0}\Vert _{\infty
}+\sup_{0\leq s\leq \tau }\Vert f^{\varepsilon }(s)\Vert _{2}+\varepsilon
^{7/2}\}.
\end{equation}
\end{lemma}

The proof of Theorem \ref{Main-Theorem} is a direct consequence of Lemmas %
\ref{l2} and \ref{linfty}.

\begin{proof}\textit{of Theorem \ref{Main-Theorem}:}
\begin{equation*}
\begin{aligned} &\frac{d}{dt}\|f^{\varepsilon }\|_{2}^{2}+\frac{c
_{0}}{2\varepsilon }\|\{\mathbf{I}-\mathbf{P}\}f^{\varepsilon
}\|_{\nu }^{2}\\ \leq & C\left\{\sqrt{\varepsilon
}\left[\|\varepsilon ^{3/2}h_{0}\|_{\infty }+\sup_{0\leq s\leq \tau
}\|f^{\varepsilon }(s)\|_{2}+\varepsilon
^{7/2}\right]+1\right\}\left(\|f^{\varepsilon
}\|_{2}^{2}+\|f^{\varepsilon }\|_{2}\right). \end{aligned}
\end{equation*}%
A simple Gronwall inequality yields%
\begin{equation*}
\|f^{\varepsilon }(t)\|_{2}+1\leq (\|f^{\varepsilon }(0)\|_{2} +1)e^{Ct\{2+%
\sqrt{\varepsilon }\|\varepsilon ^{3/2}h_{0}\|_{\infty }+\sqrt{\varepsilon }%
\sup_{0\leq s\leq \tau }\|f^{\varepsilon }(s)\|_{2}\}}\,.
\end{equation*}%
For $\varepsilon$ small, using the Taylor expansion of the
exponential function in the above inequality, we have
\begin{equation}
\| f^\eps\|_2 \leq C_1 (\| f^\eps(0)\|_2+1)\left\{ 1+\sqrt{\varepsilon }%
\|\varepsilon ^{3/2}h_{0}\|_{\infty } + \sqrt{\varepsilon
}\sup_{0\leq s\leq \tau }\|f^{\varepsilon }(s)\|_{2} \right\}\,.
\end{equation}
For $t\leq \tau ,$ letting $\varepsilon $ small, we conclude the
proof of our main theorem as:
\begin{equation*}
\sup_{0\leq t\leq \tau }\|f^{\varepsilon }(t)\|_{2}\leq C_{\tau
}\{1+\|f^{\varepsilon }(0)\|_{2}+\|\varepsilon ^{3/2}h_{0}\|_{\infty
}\}.
\end{equation*}
\end{proof}

\subsection{$L^{2}$ Estimate For $f^{\protect\varepsilon }$}

\begin{proof}\textit{of Lemma} \ref{l2}: In terms of $f^{\varepsilon },$ we obtain%
\begin{equation*}
\begin{split}
&\partial _{t}f^{\varepsilon }+v\cdot\nabla _{x}f^{\varepsilon }+\frac{1}{%
\varepsilon }{\mathcal{L}} f^{\varepsilon } \\
&=\frac{\{\partial _{t}+v\cdot \nabla _{x}\}\sqrt{\mu }}{\sqrt{\mu }}%
f^{\varepsilon }+\varepsilon ^{2}\Gamma (f^{\varepsilon
},f^{\varepsilon
})+\sum_{i=1}^5\varepsilon^{i-1}
\{\Gamma (\frac{F_{i}}{\sqrt{\mu }}%
,f^{\varepsilon })+\Gamma (f^{\varepsilon },\frac{F_{i}}{\sqrt{\mu
}})\}+\varepsilon ^{2}\bar{A}
\end{split}
\end{equation*}%
where $\bar{A}=-\frac{\{\partial _{t}+v\cdot \nabla _{x}\}F_{5}}{\sqrt{\mu }}%
+\sum_{i+j\geq 6,i\leq 5,j\leq 5}\varepsilon ^{i+j-6}\Gamma (\frac{F_{i}}{%
\sqrt{\mu }},\frac{F_{j}}{\sqrt{\mu }}).$

Taking $L^{2}$ inner product with $f^{\varepsilon }$ on both sides, since $%
\frac{\{\partial _{t}+v\cdot \nabla _{x}\}\sqrt{\mu }}{\sqrt{\mu }}$
is a cubic polynomial in $v,$ we have for any $\kappa >0$ and
$a=1/(3-\gamma)$,
\begin{equation*}
\begin{split}
&\left\langle \frac{\{\partial _{t}+v\cdot \nabla _{x}\}\sqrt{\mu }}{\sqrt{%
\mu }}f^{\varepsilon },f^{\varepsilon }\right\rangle \\
&=\int_{|v|\geq \frac{\kappa }{{\varepsilon^a }}}+\int_{|v|\leq \frac{%
\kappa }{{\varepsilon^a }}} \\
&\leq \{\|\nabla _{x}\rho \|_{2}+\|\nabla
_{x}\mathfrak{u}\|_{2}+\|\nabla _{x}T\|_{2}\}\times
\|\{1+|v|^{2}\}^{3/2}f^{\varepsilon }\mathbf{1}_{|v|\geq
\frac{\kappa }{{\varepsilon^a }}}\|_{\infty }\times \|f^{\varepsilon
}\|_{2} \\
&\;\;\;+\{\|\nabla _{x}\rho \|_{\infty }+\|\nabla
_{x}\mathfrak{u}\|_{\infty }+\|\nabla
_{x}T\|_{\infty }\}\times \|\{1+|v|^{2}\}^{3/4}f^{\varepsilon }\mathbf{1}%
_{|v|\leq \frac{\kappa }{{\varepsilon^a }}}\|_{2}^{2} \\
&\leq C_{\kappa }\varepsilon ^{2}\|h^{\varepsilon }\|_{\infty
}\|f^{\varepsilon }\|_{2}\\
&\;\;\;+C\|\{1+|v|^{2}\}^{3/4}\mathbf{P}f^{\varepsilon }%
\mathbf{1}_{|v|\leq \frac{\kappa }{{\varepsilon^a }}}\|_{2}^{2}+C\|%
\{1+|v|^{2}\}^{3/4}\{\mathbf{I}-\mathbf{P}\}f^{\varepsilon }\mathbf{1}%
_{|v|\leq \frac{\kappa }{{\varepsilon^a }}}\|_{2}^{2} \\
&\leq C_{\kappa }\varepsilon ^{2}\|h^{\varepsilon }\|_{\infty
}\|f^{\varepsilon }\|_{2}+C\|f^{\varepsilon }\|_{2}^{2}+\frac{C\kappa ^{3-\gamma}}{%
\varepsilon }\|\{\mathbf{I}-\mathbf{P}\}f^{\varepsilon }\|_{\nu
}^{2}.
\end{split}
\end{equation*}%
Here we have used the fact $\{1+|v|^{2}\}^{3/2}f^{\varepsilon }\leq
\{1+|v|^{2}\}^{\gamma-3}h^{\varepsilon },$ for $\beta \geq
3/2+(3-\gamma)$ in (\ref{h}), and the fact $\mu_M<C\mu$ in
\eqref{bound} under the assumption \eqref{assumption-T}.

By the same proof as in Lemma 2.3 of \cite{G1} and (\ref{h}),
\begin{equation*}
\varepsilon ^{2}\langle \Gamma (f^{\varepsilon },f^{\varepsilon
}),f^{\varepsilon }\rangle \leq C\varepsilon ^{2}\{\|\nu (\mu
)f^{\varepsilon }\|_{\infty }\}\|f^{\varepsilon }\|_{2}^{2}\leq C\sqrt{%
\varepsilon }\|\varepsilon ^{3/2}h^{\varepsilon }\|_{\infty
}^{{}}\|f^{\varepsilon }\|_{2}^{2}.
\end{equation*}%
Similarly, by the same proof as in Lemma 2.3 of \cite{G1} and
(\ref{h}),
\begin{equation*}
\begin{split}
&\sum_{i=1}^5\varepsilon^{i-1}\{\langle \Gamma (\frac{ F_{i}}{\sqrt{%
\mu }},f^{\varepsilon }),f^{\varepsilon }\rangle +\langle \Gamma
(f^{\varepsilon },\frac{ F_{i}}{\sqrt{%
\mu }}),f^{\varepsilon }\rangle\} \\
&\leq C\sum_{i=1}^5\varepsilon^{i-1}
\|f^{\varepsilon }\|_{\nu }^{2}\|\int_{\mathbf{R}^{3}}\frac{%
F_{i}}{\sqrt{\mu }}dv\|_{\infty } \\
&\leq C\{\|\mathbf{P}f^{\varepsilon }\|_{\nu }^{2}+\|\{\mathbf{I}-\mathbf{P}%
\}f^{\varepsilon }\|_{\nu }^{2}\} \\
&\leq C\{\|f^{\varepsilon }\|_{2}^{2}+\|\{\mathbf{I}-\mathbf{P}%
\}f^{\varepsilon }\|_{\nu }^{2}\}.
\end{split}
\end{equation*}%
Clearly, $\langle \varepsilon ^{2}\bar{A},f^{\varepsilon }\rangle
\leq
C\|f^{\varepsilon }\|_{2}.$ We therefore conclude our lemma by choosing $%
\kappa $ small.
\end{proof}

\subsection{$L^{\infty }$ Estimate For $h^{\protect\varepsilon }$}

As in \cite{Caf}, we define
\begin{equation*}
{\mathcal{L}}_{M}g=-\frac{1}{\sqrt{\mu _{M}}}\{\mathcal{Q}(\mu ,\sqrt{\mu
_{M}}g)+\mathcal{Q}(\sqrt{\mu _{M}}g,\mu )\}=\{\nu (\mu )+K\}g,
\end{equation*}%
where $Kg=K_1g-K_2g$ with
\begin{equation*}
\begin{split}
K_1g&= \int_{\mathbf{B}^3\times \mathbf{S}^2} B(\theta)|u-v|^\gamma \sqrt{%
\mu_M(u)}\frac{\mu(v)}{\sqrt{\mu_M(v)}} g(u) dud\omega \\
K_2g&= \int_{\mathbf{B}^3\times \mathbf{S}^2} B(\theta)|u-v|^\gamma {%
\mu(u^{\prime})}\frac{\sqrt{\mu_M(v^{\prime})}}{\sqrt{\mu_M(v)}}
g(v^{\prime}) dud\omega \\
&+\int_{\mathbf{B}^3\times \mathbf{S}^2} B(\theta)|u-v|^\gamma {%
\mu(v^{\prime})}\frac{\sqrt{\mu_M(u^{\prime})}}{\sqrt{\mu_M(v)}}
g(u^{\prime}) dud\omega\,.
\end{split}%
\end{equation*}
Consider a smooth cutoff function $0\leq \chi_m\leq 1$ such that for any $%
m>0 $,
\begin{equation*}
\chi_m(s)\equiv 1, \text{ for } s\leq m \,;\; \chi_m(s)\equiv 0, \text{ for }%
s\geq 2m.
\end{equation*}
Then define
\begin{equation*}
\begin{split}
K^mg&= \int_{\mathbf{B}^3\times \mathbf{S}^2} B(\theta)|u-v|^\gamma
\chi_m(|u-v|) \sqrt{\mu_M(u)}\frac{\mu(v)}{\sqrt{\mu_M(v)}} g(u) dud\omega \\
&- \int_{\mathbf{B}^3\times \mathbf{S}^2} B(\theta)|u-v|^\gamma\chi_m(|u-v|)
{\mu(u^{\prime})}\frac{\sqrt{\mu_M(v^{\prime})}}{\sqrt{\mu_M(v)}}
g(v^{\prime}) dud\omega \\
&-\int_{\mathbf{B}^3\times \mathbf{S}^2} B(\theta)|u-v|^\gamma\chi_m(|u-v|) {%
\mu(v^{\prime})}\frac{\sqrt{\mu_M(u^{\prime})}}{\sqrt{\mu_M(v)}}
g(u^{\prime}) dud\omega\,,
\end{split}%
\end{equation*}
and also define
\begin{equation*}
K^c g=K-K^m.
\end{equation*}

\begin{lemma}
\begin{equation}  \label{ks}
|K^mg(v)|\leq Cm^{3+\gamma}\nu(\mu)||g||_\infty.
\end{equation}
And $K^c g(v)=\int_{\mathbf{R}^3} l(v,{v}^{\prime})g({v}^{\prime}) d{v}%
^{\prime}$ where the kernel $l$ satisfies for some $c>0$,
\begin{equation}  \label{lestimate}
l(v,{v}^{\prime})\leq C_m\frac{\exp\{-c|v-{v}^{\prime}|^2\}}{|v-{v}%
^{\prime}| (1+|v|+|{v}^{\prime}|)^{1-\gamma}}.
\end{equation}
\end{lemma}

\begin{proof} Since $\mu\leq C\mu_M^\alpha$ for $\alpha>\frac12$ and
$|u|^2+|v|^2=|u'|^2+|v'|^2$, we first have
\[
\begin{split}
\sqrt{\mu_M(u)}\frac{\mu(v)}{\sqrt{\mu_M(v)}}&\leq C\sqrt{\mu_M(u)}
\mu_M^{\alpha-\frac12}(v), \\
{\mu(u')}\frac{\sqrt{\mu_M(v')}}{\sqrt{\mu_M(v)}}
+{\mu(v')}\frac{\sqrt{\mu_M(u')}}{\sqrt{\mu_M(v)}}&\leq
C\{\mu^{\alpha-\frac12}(u')\mu_M^{\frac12}(u)+\mu_M^{\alpha-\frac12}(v')
\mu_M^{\frac12}(u)\}.
\end{split}
\]
Since $|v-u|\leq 2m$, $\mu_M(u)\sim\mu_M(v)$ and thus
$\mu_M^{\frac12}(u)\leq C \nu(\mu)$. And since $\gamma>-3$,
\eqref{ks} follows.

To show (\ref{lestimate}), clearly the kernel for $K_{1}^c$
satisfies (\ref{lestimate}), since $\alpha>\frac12$. For $K_{2}^c,$
we can use the Carleman change of variable and apply the proof of
Lemma 1 in \cite{SG} (one can extend the result to cover all
$-3<\gamma \leq 1$).
\end{proof}

We are now ready to prove Lemma \ref{linfty}.

\begin{proof}
\textit{of Lemma} \ref{linfty}:
Letting $K_{w}g\equiv wK(\frac{g}{w}),$ from (\ref{remainder}) and (%
\ref{h}), we obtain%
\begin{equation*}
\begin{split}
&\partial _{t}h^{\varepsilon }+v\cdot \nabla _{x}h_{{}}^{\varepsilon }+%
\frac{\nu (\mu )}{\varepsilon }h^{\varepsilon }+\frac{1}{\varepsilon }%
K_{w}h^{\varepsilon } \\
=&\,\frac{\varepsilon ^{2}w}{\sqrt{\mu
_{M}}}\mathcal{Q}(\frac{h^{\varepsilon }\sqrt{\mu
_{M}}}{w},\frac{h^{\varepsilon }\sqrt{\mu _{M}}}{w})+\sum_{i=1}^5
\varepsilon^{i-1}\frac{w}{\sqrt{\mu_M}} \{\mathcal{Q}(F_i,\frac{%
h^{\varepsilon }\sqrt{\mu _{M}}}{w})+\mathcal{Q}(\frac{h^{\varepsilon }\sqrt{%
\mu _{M}}}{w}, F_i)\}\\
&+\varepsilon ^{2}\tilde{A},
\end{split}
\end{equation*}%
where $\tilde{A}=-\frac{w\{\partial _{t}+v\cdot \nabla _{x}\}F_{5}}{\sqrt{%
\mu _{M}}}+\sum_{i+j\geq 6,i\leq 5,j\leq 5}\varepsilon ^{i+j-6}\frac{w}{%
\sqrt{\mu _{M}}}\mathcal{Q}(F_{i},F_{j}).$

By Duhamel's principle, we have $h^{\varepsilon }(t,x,v)=$
\begin{equation}
\begin{split}
&\exp \{-\frac{\nu t}{\varepsilon }\}h^{\varepsilon
}(0,x-vt,v)-\int_{0}^{t}\exp \{-\frac{\nu (t-s)}{\varepsilon }\}\left( \frac{%
1}{\varepsilon }K^m_{w}h^{\varepsilon }\right) (s,x-v(t-s),v)ds   \\
&-\int_{0}^{t}\exp \{-\frac{\nu (t-s)}{\varepsilon }\}\left( \frac{%
1}{\varepsilon }K^c_{w}h^{\varepsilon }\right) (s,x-v(t-s),v)ds\\
&+\int_{0}^{t}\exp \{-\frac{\nu (t-s)}{\varepsilon }\}\left(\frac{%
\varepsilon ^{2}w}{\sqrt{\mu _{M}}}\mathcal{Q}(\frac{h^{\varepsilon }\sqrt{%
\mu _{M}}}{w},\frac{h^{\varepsilon }\sqrt{\mu _{M}}}{w}%
)\right)(s,x-v(t-s),v)ds   \\
&+\int_{0}^{t}\exp \{-\frac{\nu (t-s)}{\varepsilon
}\}\left(\sum_{i=1}^5 \varepsilon^{i-1}\frac{w}{\sqrt{\mu_M}}
\mathcal{Q}(F_i,\frac{h^{\varepsilon
}\sqrt{\mu _{M}}}{w})\right)(s,x-v(t-s),v)ds  \\
&+\int_{0}^{t}\exp \{-\frac{\nu (t-s)}{\varepsilon
}\}\left(\sum_{i=1}^5
\varepsilon^{i-1}\frac{w}{\sqrt{\mu_M}} \mathcal{Q}(\frac{h^{\varepsilon }%
\sqrt{\mu _{M}}}{w}, F_i)\right)(s,x-v(t-s),v)ds  \\
&+\int_{0}^{t}\exp \{-\frac{\nu (t-s)}{\varepsilon }\}\varepsilon ^{2}%
\tilde{A}(s,x-v(t-s),v)ds.
\end{split} \label{duhamel}
\end{equation}
First note that
\begin{equation*}
\begin{split}
\nu (\mu ) \sim\int |v-u|^\gamma \mu du\sim (1+|v|)^\gamma\rho
(t,x)\sim
\nu _{M}(v), \\
\int_{0}^{t}\exp \{-\frac{\nu (\mu )(t-s)}{\varepsilon }\}\nu (\mu
)ds \leq c\int_{0}^{t}\exp \{-\frac{c\nu _{M}(t-s)}{\varepsilon
}\}\nu _{M}ds=O(\varepsilon ).
\end{split}
\end{equation*}
Then from (\ref{ks}), the second term in \eqref{duhamel} is bounded
by
\begin{equation*}
Cm^{3+\gamma }\int_{0}^{t}\exp \{-\frac{\nu (t-s)}{\varepsilon
}\}\nu ds\sup_{0\leq t\leq \tau }||h^{\varepsilon }(t)||_{\infty
}\leq Cm^{3+\gamma}\varepsilon\sup_{0\leq t\leq \tau
}||h^{\varepsilon }(t)||_{\infty}.
\end{equation*}%
By $\mu _{M}\leq C\mu ,$ and since $|\frac{w}{\sqrt{\mu _{M}}}
Q(\frac{h^{\varepsilon }%
\sqrt{\mu _{M}}}{w},\frac{h^{\varepsilon }\sqrt{\mu _{M}}}{w})|\leq
C\nu (\mu )\|h^{\varepsilon }\|_{\infty }^{2}$  from the same proof
as in Lemma 10 of \cite{G2}, the third line in (\ref{duhamel}) is
bounded by
\begin{equation}
\begin{split}
&C\varepsilon ^{2}\int_{0}^{t}\exp \{-\frac{\nu (\mu )(t-s)}{\varepsilon }%
\}\nu (\mu )\|h^{\varepsilon }(s)\|_{\infty }^{2}ds   \label{h2} \\
\leq\;& C\varepsilon ^{3}\sup_{0\leq s\leq t}\|h^{\varepsilon
}(s)\|_{\infty }^{2}.
\end{split}
\end{equation}%
 From the same proof as in Lemma 10 of \cite{G2} again,
\begin{equation*}
\begin{split}
&\sum_{i=1}^5 \varepsilon^{i-1}\frac{w}{\sqrt{\mu_M}}\{ \mathcal{Q}(F_i,\frac{%
h^{\varepsilon }\sqrt{\mu _{M}}}{w})+\mathcal{Q}(\frac{h^{\varepsilon }\sqrt{%
\mu _{M}}}{w},F_i)\}\\
& \leq \nu _{M}(v)\|h^{\varepsilon }\|_\infty \|\frac{w}{%
\sqrt{\mu _{M}}}\sum_{i=1}^5 \varepsilon^{i-1}F_{i}\|_{\infty }
\end{split}
\end{equation*}%
so that the fourth and fifth lines in (\ref{duhamel}) are bounded by
\begin{equation}
C\int_{0}^{t}\exp \{-\frac{\nu (\mu )(t-s)}{\varepsilon }\}\nu
_{M}(v)\|h^{\varepsilon }(s)\|_{\infty }ds\leq C\varepsilon
\sup_{0\leq s\leq t}\|h^{\varepsilon }(s)\|_{\infty }.  \label{h1}
\end{equation}
The last line in (\ref{duhamel}) is clearly bounded by $C\varepsilon
^{3}.$

We shall mainly concentrate on the third term in the right hand side of (%
\ref{duhamel}). Let $l_{w}(v,v^{\prime })$ be the corresponding
kernel associated with $K_{w}^c$. Recalling \eqref{lestimate}, we
have
\begin{equation}
|l_{w}(v,v^{\prime })|\leq \frac{Cw(v^{\prime })\exp
\{-c|v-v^{\prime
}|^{2}\}}{|v-v^{\prime }|w(v)(1+|v|+|v^{\prime }|)^{1-\gamma }}\leq \frac{%
C\exp \{-\widetilde{c}|v-v^{\prime }|^{2}\}}{|v-v^{\prime
}|(1+|v|+|v^{\prime }|)^{1-\gamma }} \label{k}
\end{equation}
with a smaller $\widetilde{c}>0.$ Since $\nu (\mu )$ $\sim \nu
_{M},$ we bound the second line in \eqref{duhamel} by
\begin{equation*}
\frac{1}{\varepsilon }\int_{0}^{t}\exp \{-\frac{\nu (t-s)}{\varepsilon }%
\}\int_{\mathbf{R}^{3}}|l_{w}(v,v^{\prime })h^{\varepsilon
}(s,x-v(t-s),v^{\prime })|dv^{\prime }ds,
\end{equation*}%
We now use (\ref{duhamel}) again to evaluate $h^{\varepsilon }$. By
(\ref{h2}) and (\ref{h1}), we can
bound the above by%
\begin{equation}
\begin{split}
&\frac{1}{\varepsilon }\int_{0}^{t}\exp \{-\frac{\nu (t-s)}{\varepsilon }%
\}\sup_{v}\int_{\mathbf{R}^{3}}|l_{w}(v,v^{\prime })|dv^{\prime }\exp \{-%
\frac{\nu s}{\varepsilon }\}h^{\varepsilon }(0,x-v(t-s)-v^{\prime
}s,v^{\prime })ds  \\
&+\frac{1}{\varepsilon ^{2}}\int_{0}^{t}\exp \{-\frac{\nu
(t-s)}{\varepsilon }\}\int_{\mathbf{R}^{3}}|l_{w}(v,v^{\prime
})|\\
&\times\int_{0}^{s}\exp \{-\frac{\nu (v^{\prime
})(s-s_{1})}{\varepsilon }\}|\{K^{m}h^{\varepsilon
}\}(s_1,x-v(t-s)-v^{\prime
}(s-s_{1}),v^{\prime})|dv^{\prime } ds_1ds \\
&+\frac{1}{\varepsilon ^{2}}\int_{0}^{t}\exp \{-\frac{\nu (t-s)}{%
\varepsilon }\}\int_{\mathbf{R}^{3}\times \mathbf{R}^{3}}|l_{w}(v,v^{%
\prime })l_{w}(v^{\prime },v^{\prime \prime })   \\
&\times \int_{0}^{s}\exp \{-\frac{\nu (v^{\prime })(s-s_{1})}{\varepsilon }%
\}h^{\varepsilon }(s_1,x-v(t-s)-v^{\prime }(s-s_{1}),v^{\prime
\prime
})|dv^{\prime }dv^{\prime \prime }ds_{1}ds \\
&+\frac{C}{\varepsilon }\int_{0}^{t}\exp \{-\frac{\nu (t-s)}{\varepsilon }%
\}ds\times \int_{\mathbf{R}^{3}}|l_{w}(v,v^{\prime })|dv^{\prime
}\times \{\varepsilon ^{3}\sup_{0\leq s\leq t}\|h^{\varepsilon
}(s)\|_{\infty }^{2}\} \\
&+\frac{C}{\varepsilon }\int_{0}^{t}\exp \{-\frac{\nu (t-s)}{\varepsilon }%
\}ds\times\int_{\mathbf{R}^{3}}|l_{w}(v,v^{\prime })|dv^{\prime
}\times \{\varepsilon \sup_{0\leq s\leq t}\|h^{\varepsilon
}(s)\|_{\infty }\}
\\
&+\frac{C}{\varepsilon }\int_{0}^{t}\exp \{-\frac{\nu (t-s)}{\varepsilon }%
\}ds\times\int_{\mathbf{R}^{3}}|l_{w}(v,v^{\prime })|dv^{\prime
}\times \{\varepsilon ^{2}\sup_{0\leq s\leq t}\|\tilde{A}\|_{\infty
}\}.
\end{split}\label{double}
\end{equation}
By (\ref{ks}), the second term is bounded as follows:
\begin{equation*}
\begin{split}
&\frac{Cm^{\gamma +3}}{\varepsilon ^{2}}\sup_{0\leq \tau\leq t
}\|h^{\varepsilon }(\tau)\|_{\infty }\int_{0}^{t}\exp \{-\frac{\nu (t-s)}{%
\varepsilon }\}\int_{\mathbf{R}^{3}}|l_{w}(v,v^{\prime
})|\\
&\quad\times \int_{0}^{s}\exp \{-\frac{\nu (v^{\prime
})(s-s_{1})}{\varepsilon }\}\nu (v^{\prime
})dv^{\prime }ds_1ds \\
&\leq \frac{Cm^{\gamma +3}}{\varepsilon ^{{}}}\sup_{0\leq \tau\leq t
}\|h^{\varepsilon }(\tau )\|_{\infty }\int_{0}^{t}\exp \{-\frac{\nu (t-s)}{%
\varepsilon }\}\int_{\mathbf{R}^{3}}|l_{w}(v,v^{\prime })|dv^{\prime }ds \\
&=Cm^{\gamma +3}\sup_{0\leq \tau\leq t}\|h^{\varepsilon }(\tau
)\|_{\infty },
\end{split}
\end{equation*}
where we have used the fact
 $\int_{\mathbf{R}^{3}}|l_{w}(v,v^{\prime })|dv^{\prime }<\nu (v)$ from (%
\ref{k}).  Similar arguments for other terms except the third term
yield the following bound:%
\begin{equation*}
C\{\|h^{\varepsilon }(0)\|_{\infty }+\varepsilon ^{3}\sup_{0\leq
s\leq t}\|h^{\varepsilon }(s)\|_{\infty }^{2}+\varepsilon
\sup_{0\leq s\leq t}\|h^{\varepsilon }(s)\|_{\infty
}^{{}}+C\varepsilon ^{3}\}.
\end{equation*}

We now concentrate on the third term in (\ref{double}), which will
be estimated as in the proof of Theorem 20 in \cite{G2}.

\textbf{CASE 1:} For $|v|\geq N.$ By (\ref{k}),
\begin{equation*}
\int_{\mathbf{R}^{3}}|l_{w}(v,v^{\prime })|dv^{\prime }\leq C\frac{\nu (v)}{N%
}\text{ \ and \ }\int_{\mathbf{R}^{3}}|l_{w}(v^{\prime },v^{\prime
\prime })|dv^{\prime \prime }\leq C\nu (v^{\prime })
\end{equation*}%
and thus we have the following bound
\begin{equation*}
\begin{split}
&\frac{C}{\varepsilon }\sup_{0\leq s\leq t}\|h^{\varepsilon
}(s)\|_{\infty
}\int_{0}^{t}\int_{\mathbf{R}^{3}}\exp \{-\frac{\nu (v)(t-s)}{\varepsilon }%
\}|l_{w}(v,v^{\prime })|\\
&\quad\quad\quad\times \int_{0}^{s}\exp \{-\frac{\nu (v^{\prime })(s-s_{1})%
}{\varepsilon }\}\frac{\nu (v^{\prime })}{\varepsilon
}ds_{1}dv^{\prime }ds\\
&\leq \frac{C}{N}\sup_{0\leq s\leq t}\|h^{\varepsilon }(s)\|_{\infty
}.
\end{split}
\end{equation*}

\textbf{CASE 2:}\textit{\ }For $|v|\leq N,$ $|v^{\prime }|\geq 2N,$ or $%
|v^{\prime }|\leq 2N$, $|v^{\prime \prime }|\geq 3N.$ Notice that we
have either $|v^{\prime }-v|\geq N$ or $|v^{\prime }-v^{\prime
\prime }|\geq N,$ and either one of the following is valid
correspondingly for some small $\eta >0$:
\begin{equation}
\begin{split}
|l_{w}(v,v^{\prime })|&\leq e^{-\frac{\eta
}{8}N^{2}}|l_{w}(v,v^{\prime
})e^{\frac{\eta }{8}|v-v^{\prime }|^{2}}|,\\
|l_{w}(v^{\prime },v^{\prime \prime })|&\leq e^{-\frac{\eta }{8}%
N^{2}}|l_{w}(v^{\prime },v^{\prime \prime })e^{\frac{\eta
}{8}|v^{\prime }-v^{\prime \prime }|^{2}}|.  \label{kwe}
\end{split}
\end{equation}%
From \eqref{k}, we obtain
\begin{equation*}
\int |l_{w}(v,v^{\prime })e^{\frac{\eta }{8}|v-v^{\prime
}|^{2}}|dv^{\prime
}\leq C\nu (v),\text{ and }\int |l_{w}(v^{\prime },v^{\prime \prime })e^{%
\frac{\eta }{8}|v^{\prime }-v^{\prime \prime }|^{2}}|dv^{\prime
\prime }\leq C\nu (v^{\prime }).
\end{equation*}
We use (\ref{kwe}) to combine the cases of $|v^{\prime }-v|\geq N$ or $%
|v^{\prime }-v^{\prime \prime }|\geq N$ as:
\begin{equation}
\begin{split}
&\int_{0}^{t}\int_{0}^{s}\left\{ \int_{|v|\leq N,|v^{\prime }|\geq
2N}+\int_{|v^{\prime }|\leq 2N,|v^{\prime \prime }|\geq
3N}\right\}\\
&\leq \frac{C_{\eta }}{\varepsilon ^{2}}e^{-\frac{\eta }{8}%
N^{2}}\sup_{0\leq s\leq t}\|h^{\varepsilon }(s)\|_{\infty }
\int_{0}^{t}\int_{0}^{s}\int |l_w(v,v')|
\exp \{-\frac{\nu (v)(t-s)}{\varepsilon }%
\}\\
&\quad\quad\quad\quad\quad\quad\quad\quad\quad\quad\quad\quad\quad\quad\exp
\{-\frac{\nu(v^{\prime})(s-s_{1})}{\varepsilon
}\} \nu(v')dv'ds_1ds  \\
&\leq C_{\eta }e^{-\frac{\eta }{8}N^{2}}\sup_{0\leq s\leq
t}\{\|h^{\varepsilon }(s)\|_{\infty }\}.  \label{inflowstep3}
\end{split}
\end{equation}

\textbf{CASE 3a:}\textit{\ \thinspace } $|v|\leq N,$ $|v^{\prime
}|\leq 2N,|v^{\prime \prime }|\leq 3N.$ This is the last remaining
case
because if $|v^{\prime }|>2N,$ it is included in Case 2; while if $%
|v^{\prime \prime }|>3N,$ either $|v^{\prime }|\leq 2N$ or
$|v^{\prime }|\geq 2N$ are also included in Case 2. We further
assume
that $s-s_1\leq \varepsilon \kappa ,$ for $%
\kappa >0$ small. We bound the third term in (\ref{double}) by
\begin{equation}
\begin{split}
&\frac{1}{\varepsilon ^{2}}\int_{0}^{t}\int_{s-\varepsilon \kappa
}^{s}C\exp \{-\frac{\nu (v)(t-s)}{\varepsilon }\}\exp \{-\frac{%
\nu(v^{\prime})(s-s_{1})}{\varepsilon }\}\|h^{\varepsilon
}(s_1)\|_{\infty
}ds_{1} ds \\
&\leq C_N\sup_{0\leq s\leq t}\{\|h^{\varepsilon }(s)\|_{\infty
}\}\times
\frac{1}{\varepsilon }\int_{0}^{t}\exp \{-\frac{\nu (v)(t-s)}{\varepsilon }%
\}ds\times \int_{s-\varepsilon \kappa }^{s}\frac{1}{\varepsilon
}ds_1
\\
&\leq \kappa C_N\sup_{0\leq s\leq t}\{\|h^{\varepsilon
}(s)\|_{\infty }\}. \label{inflowstep1}
\end{split}
\end{equation}

\textbf{CASE 3b:}  $|v|\leq N,$ $|v^{\prime }|\leq 2N,|v^{\prime
\prime }|\leq 3N,$ and $s-s_1\geq \varepsilon .$
We now can bound the third term in (%
\ref{double}) by
\begin{equation*}
\begin{split}
C\int_{0}^{t}\int_{B}\int_{0}^{s-\varepsilon \kappa }e^{-\frac{\nu (v)(t-s)}{%
\varepsilon }}e^{-\frac{\nu(v^{\prime })(s-s_{1})}{\varepsilon }%
}&|l_{M,w}(v,v^{\prime })l_{M,w}(v^{\prime },v^{\prime \prime
})\\
&h^{\varepsilon }(s_1,x_{1}-(s-s_1)v^{\prime },v^{\prime \prime })|
ds_{1}dv^{\prime }dv^{\prime \prime }ds
\end{split}
\end{equation*}%
where $B=\{|v^{\prime }|\leq 2N,$ $|v^{\prime \prime }|\leq 3N\}.$ By (\ref%
{k}), $l_{w}(v,v^{\prime })$ has possible integrable singularity of $\frac{%
1}{|v-v^{\prime }|},$ we can choose $l_{N}(v,v^{\prime })$ smooth
with compact support such that
\begin{equation}
\sup_{|p|\leq 3N}\int_{|v^{\prime }|\leq 3N}|l_{N}(p,v^{\prime
})-l_{w}(p,v^{\prime })|dv^{\prime }\leq \frac{1}{N}.
\label{approximate}
\end{equation}%
Splitting
\begin{equation*}
\begin{split}
l_{w}(v,v^{\prime })l_{w}(v^{\prime },v^{\prime \prime })
&=\{l_{w}(v,v^{\prime })-l_{N}(v,v^{\prime })\}l_{w}(v^{\prime
},v^{\prime \prime }) \\
&+\{l_{w}(v^{\prime },v^{\prime \prime })-l_{N}(v^{\prime
},v^{\prime \prime })\}l_{N}(v,v^{\prime })+l_{N}(v,v^{\prime
})l_{N}(v^{\prime },v^{\prime \prime }),
\end{split}
\end{equation*}%
we can use such an approximation (\ref{approximate}) to bound the above $%
s_{1},s$ integration by
\begin{equation}
\begin{split}\label{inflowstep41}
\frac{C}{N}\sup_{0\leq s\leq t}\{\|h^{\varepsilon }(s)\|_{\infty
}\}\times \{ \sup_{|v^{\prime }|\leq 2N}\int |l_{w}(v^{\prime
},v^{\prime \prime })|dv^{\prime \prime }+\sup_{|v|\leq 2N}\int
|l_{N}(v,v^{\prime
})|dv^{\prime } \}   \\
+C\int_{0}^{t}\int_{B}\int_{0}^{s-\varepsilon \kappa }e^{-\frac{\nu
(v)(t-s)}{\varepsilon }}e^{-\frac{\nu(v^{\prime
})(s-s_{1})}{\varepsilon}}
|l_{N}(v,v^{\prime })l_{N}(v^{\prime },v^{\prime \prime})\\
h^{\varepsilon }(s_1,x_{1}-(s-s_1)v^{\prime },v^{\prime \prime })|
ds_{1}dv^{\prime }dv^{\prime \prime }ds.
\end{split}
\end{equation}%
Since $l_{N}(v,v^{\prime })l_{N}(v^{\prime },v^{\prime \prime })$ is
bounded, we first integrate over $v^{\prime }$ to get
\begin{equation*}
\begin{split}
&C_{N}\int_{|v^{\prime }|\leq 2N}|h^{\varepsilon
}(s_1,x_{1}-(s-s_1)v^{\prime },v^{\prime \prime })|dv^{\prime } \\
&\leq C_{N}\left\{ \int_{|v^{\prime }|\leq 2N}\mathbf{1}_{\Omega
}(x_{1}-(s-s_1)v^{\prime })|h^{\varepsilon
}(s_1,x_{1}-(s-s_1)v^{\prime
},v^{\prime \prime })|^{2}dv^{\prime }\right\} ^{1/2} \\
&\leq \frac{C_{N}}{\kappa ^{3/2}\varepsilon ^{3/2}}\left\{
\int_{|y-x_{1}|\leq (s-s_1)3N}|h^{\varepsilon }(s_1,y,v^{\prime
\prime
})|^{2}dy\right\} ^{1/2} \\
&\leq \frac{C_{N}\{(s-s_1)^{3/2}+1\}}{\kappa ^{3/2}\varepsilon ^{3/2}}%
\left\{ \int_{\mathbf{\Omega }}|h^{\varepsilon }(s_1,y,v^{\prime
\prime })|^{2}dy\right\} ^{1/2}.
\end{split}
\end{equation*}%
Here we have made a change of variable $y=x_{1}-(s-s_1)v^{\prime },$
and for
$s-s_1\geq \kappa \varepsilon ,$ $|\frac{dy}{dv^{\prime }}|\geq \frac{1}{%
\kappa ^{3}\varepsilon ^{3}}.$ In the case of $\Omega
=\mathbf{R}^{3},$ the factor $\{(s-s_1)^{3/2}+1\}$ is not needed. By
(\ref{h}) and (\ref{f}), we
then further control the last term in (\ref{inflowstep41}) by:%
\begin{equation*}
\begin{split}
&\frac{C_{N,\kappa }}{\varepsilon
^{7/2}}\int_{0}^{t}\int_{0}^{s-\kappa \varepsilon }e^{-\frac{\nu
(v)(t-s)}{\varepsilon }}e^{-\frac{\nu(v^{\prime
})(s-s_{1})}{\varepsilon }}\{(s-s_1)^{3/2}+1\}\\
&\quad\quad\quad\quad\quad\int_{|v^{\prime \prime }|\leq 3N}\left\{
\int_{\mathbf{\Omega }}|h^{\varepsilon }(s_1,y,v^{\prime \prime
})|^{2}dy\right\} ^{1/2}dv^{\prime \prime }ds_1ds   \\
&\leq \frac{C_{N,\kappa }}{\varepsilon ^{7/2}}\int_{0}^{t}\int_{0}^{s-%
\kappa \varepsilon }e^{-\frac{\nu (v)(t-s)}{\varepsilon }}e^{-\frac{%
\nu(v^{\prime })(s-s_{1})}{\varepsilon
}}\{(s-s_1)^{3/2}+1\}\\
&\quad\quad\quad\quad\quad\left\{ \int_{|v^{\prime \prime }|\leq
3N}\int_{\Omega }|f^{\varepsilon }(s_1,y,v^{\prime \prime
})|^{2}dydv^{\prime \prime }\right\} ^{1/2}ds_1ds
\\
&\leq \frac{C_{N,\kappa }}{\varepsilon ^{3/2}}\sup_{0\leq s\leq
t}\|f^{\varepsilon }(s)\|_{2}.
\end{split}
\end{equation*}

In summary, we have established, for any $\kappa >0$ and large $N>0,$%
\begin{equation*}
\begin{split}
\sup_{0\leq s\leq t}\{\varepsilon ^{3/2}\|h^{\varepsilon
}(s)\|_{\infty }\} \leq \{Cm^{\gamma+3}+C_{N,m}\kappa
+\frac{C_{m}}{N}\}\sup_{0\leq s\leq t}\{\varepsilon
^{3/2}\|h^{\varepsilon }(s)\|_{\infty }\} +\varepsilon
^{7/2}C\\+C_{\varepsilon ,N}\|\varepsilon ^{3/2}h_{0}\|_{\infty }
+\sqrt{\varepsilon }C\sup_{0\leq s\leq t}\{\varepsilon
^{3/2}\|h^{\varepsilon }(s)\|_{\infty }\}^{2}+C_{m,N,\kappa
}\sup_{0\leq s\leq t}\|f^{\varepsilon }(s)\|_{2}.
\end{split}
\end{equation*}%
For sufficiently small $\varepsilon>0$,
first choosing $m$ small,  then $N$ sufficiently large, and finally letting $%
\kappa $ small  so that $\{Cm^{\gamma +3}+C_{N,m}\kappa +\frac{C_{m}}{N}\}<%
\frac{1}{2}$, we get
\begin{equation*}
\sup_{0\leq s\leq \tau }\{\varepsilon ^{3/2}\|h^{\varepsilon
}(s)\|_{\infty }\}\leq C\{\|\varepsilon ^{3/2}h_{0}\|_{\infty
}+\sup_{0\leq s\leq \tau }\|f^{\varepsilon }(s)\|_{2}+\varepsilon
^{7/2}\}
\end{equation*}%
and we conclude our proof.
\end{proof}

\section{Acoustic Limit}

\label{3}

\subsection{Compressible Euler and Acoustic systems}

We note that the acoustic system \eqref{acoustic-system} is essentially
linear wave equations. Thus the well-posedness follows from the linear
theory of wave equations: for the given initial data $(\sigma^0,\,u^0,\,%
\theta^0)\in H^s$, there exist global-in-time classical solutions $%
(\sigma,\,u,\,\theta) \in C([0, \infty)\,; H^s)$ to the acoustic system %
\eqref{acoustic-system}. In particular, we have the following energy
estimates: for each $s\geq 0$,
\begin{equation}  \label{acoustic}
\|(\sigma,\,u\,,\sqrt{\tfrac{3}{2}}\theta)(t)\|^2_{H^s} =\|(\sigma^0,\,u^0,\,%
\sqrt{\tfrac{3}{2}}\theta^0)\|^2_{ H^s}, \text{ for all }t\geq 0\,.
\end{equation}

On the other hand, classical solutions to compressible Euler equations exist
for only finite time \cite{Sid}. Since the properties of solutions play an
important role in our argument, we present the existence result of smooth
solutions to compressible Euler system. Normalizing $R\equiv 1$ in the
equation of state \eqref{ES}, we can write the invisid flow equations in
variables $\rho,\,\mathfrak{u},\,T$ as follows:
\begin{equation}  \label{Euler1}
\begin{split}
\partial_t\rho +(\mathfrak{u}\!\cdot\!\nabla)\rho+\rho\nabla\!\cdot\!
\mathfrak{u}=0\,, \\
\rho\partial_t \mathfrak{u}+\rho (\mathfrak{u}\!\cdot\!\nabla)\mathfrak{u}
+\rho\nabla T+T\nabla\rho=0\,, \\
\partial_t T+(\mathfrak{u}\!\cdot\!\nabla)T+\tfrac23 T\nabla\!\cdot\!
\mathfrak{u}=0\,.
\end{split}%
\end{equation}
It is a classical result from the theory of symmetric hyperbolic system that
the lifespan of smooth solutions to 3D compressible Euler equations with
smooth initial data, which are a small perturbation of amplitude $\delta$
from a constant state, is at least $O(\delta^{-1})$. We summarize the result
in the following lemma:

\begin{lemma}
\label{Life-Span} Consider the compressible Euler system \eqref{Euler1} with
initial data:
\begin{equation}  \label{pert-initial}
\rho^0=1+\delta \sigma^0 \,,\quad\! \mathfrak{u}^0=\delta u^0\,,\quad\!
T^0=1+\delta \theta^0\,,
\end{equation}
for any given $(\sigma^0, u^0, \theta^0) \in H^s$ with $s>\frac52 $. Choose $%
\delta_1>0$ so that for any $0<\delta\leq \delta_1$, the positivity of $%
\rho^0$ and $T^0$ is guaranteed. Then, for each $0<\delta\leq \delta_1$,
there is a family of classical solutions $(\rho^\delta, \mathfrak{u}^\delta,
T^\delta) \in C([0, \tau^\delta]\,; H^s)\cap C^1([0, \tau^\delta]\,;
H^{s-1}) $ of the Euler equations \eqref{Euler1} such that $\rho^\delta>0$, $%
T^\delta>0$, and the following estimates hold:
\begin{equation}  \label{pert-euler}
\|(\rho^\delta, \mathfrak{u}^\delta, T^\delta)-(1,0,1)\|_{C([0,
\tau^\delta]; H^s)\cap C^1([0, \tau^\delta]; H^{s-1})} \leq C_0\,.
\end{equation}
Furthermore, the lifespans $\tau^\delta$ have the following lower bound
\begin{equation*}
\tau^\delta > \frac{C_1}{\delta}\,.
\end{equation*}
Here the constants $C_0\,, C_1$ are independent of $\delta$, depend only on
the $H^s\mbox{-}$norm of $(\sigma^0, u^0, \theta^0)$.
\end{lemma}

We omit the proof of Lemma \ref{Life-Span} (see \cite{Fried,K,M}).

Now, for any given $\tau >0$ and given acoustic initial data $(\sigma^0,
u^0, \theta^0) \in H^s$, we define
\begin{equation}  \label{delta-1}
\delta_1=\frac{C_1}{\tau}\,.
\end{equation}
Thus the lifespan of the solutions $(\rho^\delta, \mathfrak{u}^\delta,
T^\delta)$ of the compressible Euler equations constructed in Lemma \ref%
{Life-Span} have a uniform lower bound
\begin{equation*}
\tau^\delta > \frac{C_1}{\delta} > \frac{C_1}{\delta_1} = \tau\,.
\end{equation*}
From now on, we consider the solutions of the compressible Euler system on
an arbitrary finite time interval $[0\,, \tau]$ and we fix $\delta_1>0$ as
in \eqref{delta-1}.

Next, we derive a refined estimate of two solutions to compressible Euler
and acoustic systems. In order to do so, we first introduce the following
difference variables $(\sigma^\delta_d,u^\delta_d, \theta^\delta_d)$ that
are given by the second order perturbation in $\delta$ of Euler solutions:
\begin{equation}  \label{expansion}
\delta^2\sigma^\delta_d\equiv \rho^\delta-1-\delta\sigma\,,\quad \delta^2
u^\delta_d\equiv \mathfrak{u}^\delta-\delta u\,,\quad
\delta^2\theta^\delta_d\equiv T^\delta-1-\delta\theta\,.
\end{equation}

\begin{lemma}
\label{keylemma} Let $\tau >0.$ Let $(\sigma ^{0},\,u^{0},\,\theta ^{0})\in
H^{s}$ in \eqref{acoustic-system} with corresponding  acoustic solution $%
(\sigma ,\,u,\,\theta )$. Let $(\rho ^{\delta },\mathfrak{u}^{\delta
},T^{\delta })$ be the Euler solutions of \eqref{Euler1} with the
corresponding initial data \eqref{pert-initial} constructed in Lemma \ref%
{Life-Span}. Then for all $0<\delta \leq \delta _{0}$ and for $s\geq 3$,
there exists a constant $C_{2}>0$ only depending on $\tau $ and $H^{s+1}%
\mbox{-}$norm of $(\sigma ^{0},u^{0},\theta ^{0})$ such that
\begin{equation}
\Vert (\sigma _{d}^{\delta }\,,u_{d}^{\delta }\,,\theta _{d}^{\delta })\Vert
_{H^{s}}\leq C_{2}.  \label{delta^2}
\end{equation}
\end{lemma}

Lemma \ref{keylemma} verifies that the acoustic system is the linearization
about the constant state of the compressible Euler system:
\begin{equation}
\sup_{0\leq t\leq \tau }\Vert (\rho ^{\delta }-1-\delta \sigma ,\;\mathfrak{u%
}^{\delta }-\delta u,\;T^{\delta }-1-\delta \theta )\Vert _{H^{s}}\leq
C_{2}\delta ^{2}.  \label{refine-estimate}
\end{equation}%
In addition, by the estimate \eqref{delta^2} for $s\geq 3$ and Sobolev
embedding theorem, we obtain the uniform point-wise estimates of the
difference variables $(\sigma _{d}^{\delta },u_{d}^{\delta },\theta
_{d}^{\delta })$.

\begin{proof}\textit{of Lemma \ref{keylemma}:}
Rewrite \eqref{expansion} as
\begin{equation}\label{perturbation}
\rho^\delta=1+\delta\sigma+\delta^2\sigma^\delta_d\,,\quad
\mathfrak{u}^\delta=\delta u+\delta^2 u^\delta_d\,,\quad
 T^\delta=1+\delta\theta+\delta^2\theta^\delta_d\,,
\end{equation}
and plug into \eqref{Euler1} to get:
\[
\begin{split}
\partial_t[\delta\sigma+\delta^2\sigma^\delta_d]+(\delta
u+\delta^2 u^\delta_d) \cdot\nabla(\delta\sigma+
\delta^2\sigma^\delta_d) + (1+\delta\sigma+
\delta^2\sigma^\delta_d)\nabla\cdot(\delta
u+\delta^2 u^\delta_d)&=0\\
(1+\delta\sigma+\delta^2\sigma^\delta_d)\partial_t[\delta u+\delta^2
u^\delta_d]+ (1+\delta\sigma+\delta^2\sigma^\delta_d)[(\delta
u+\delta^2 u^\delta_d)\cdot\nabla] (\delta u+\delta^2 u^\delta_d)
\quad&\\
+ (1+\delta\sigma+\delta^2\sigma^\delta_d)\nabla
(\delta\theta+\delta^2\theta^\delta_d)
+(1+\delta\theta+\delta^2\theta^\delta_d)\nabla
(\delta\sigma+\delta^2\sigma^\delta_d)&=0\\
\partial_t[\delta\theta+\delta^2\theta^\delta_d]+ (\delta
u+\delta^2 u^\delta_d)
\cdot\nabla(\delta\theta+\delta^2\theta^\delta_d)+
\tfrac23(1+\delta\theta+\delta^2\theta^\delta_d)\nabla\cdot(\delta
u+\delta^2 u^\delta_d)&=0
\end{split}
\]
Coefficients of $\delta$ in each equation form the acoustic system
\eqref{acoustic-system} in $(\sigma, u,\theta)$, which is indeed the
acoustic solution by the assumption. Hence, the remaining terms are
at least of order $O(\delta^2)$. For instance, the continuity
equation reduces to
\[
\delta^2[\partial_t\sigma^\delta_d+(u+\delta
u^\delta_d)\cdot\nabla\sigma+\underbrace{(\delta u+\delta^2
u^\delta_d)}_{(a)}\cdot\nabla\sigma^\delta_d +
(\sigma+\delta\sigma^\delta_d)\nabla\cdot u+
\underbrace{(1+\delta\sigma+
\delta^2\sigma^\delta_d)}_{(b)}\nabla\cdot u_d^\delta]=0\,.
\]
Use \eqref{perturbation} to replace $(a)$ and $(b)$ by
$\mathfrak{u}^\delta$ and $T^\delta$ respectively. Then the equation
can be written as
\[
\partial_t\sigma^\delta_d+(\mathfrak{u}^\delta\cdot\nabla)\sigma^\delta_d
+\rho^\delta\nabla\cdot u^\delta_d + \delta[\nabla\sigma\cdot
u^\delta_d+(\nabla\cdot u)\sigma^\delta_d]+\nabla\cdot(\sigma
u)=0\,.
\]
Similarly, one can deduce that $(\sigma^\delta_d\,,u^\delta_d\,,
\theta^\delta_d)$ satisfies the following linear system of
equations:
\begin{equation}\label{key}
\begin{split}
&\partial_t\sigma^\delta_d+(\mathfrak{u}^\delta\cdot\nabla)\sigma^\delta_d
+\rho^\delta\nabla\cdot u^\delta_d + \delta[\nabla\sigma\cdot
u^\delta_d+(\nabla\cdot
u)\sigma^\delta_d]=-\nabla\cdot(\sigma u)\\
&\rho^\delta \partial_t
u^\delta_d+\rho^\delta(\mathfrak{u}^\delta\cdot\nabla)u^\delta_d
+\rho^\delta\nabla\theta^\delta_d+T^\delta\nabla\sigma^\delta_d
+\delta [(\partial_tu) \sigma^\delta_d+\rho^\delta(u^\delta_d\cdot
\nabla)u+\theta^\delta_d\nabla\sigma+\sigma^\delta_d\nabla\theta]\\
&\;\;=-\sigma\partial_tu-\rho^\delta(u\cdot\nabla)u-\nabla(\sigma\theta)\\
&\partial_t\theta^\delta_d+(\mathfrak{u}^\delta\cdot\nabla)\theta^\delta_d
+\tfrac23 T^\delta\nabla\cdot u^\delta_d +\delta[\nabla\theta\cdot
u^\delta_d+\tfrac23(\nabla\cdot
u)\theta^\delta_d]=-u\cdot\nabla\theta-\tfrac23\theta\nabla\cdot u
\end{split}
\end{equation}
The point here is that the above system is linear in
$\sigma^\delta_d\,,u^\delta_d\,, \theta^\delta_d$, although the
coefficients may depend on $\rho^\delta,\,
\mathfrak{u}^\delta,\,T^\delta$ of compressible Euler system and
$\sigma,\,u,\,\theta$ of acoustic system. However, we already know
that these coefficients are smooth at least up to time $\tau$  and
they have Sobolev energy bounds \eqref{acoustic} and
\eqref{pert-euler}. Note that the system \eqref{key} can be written
as a symmetric system with the corresponding symmetrizer $A_0$:
\begin{equation}\label{symmetric}
A_0\partial_t U_d +\sum_{i=1}^3 A_i\partial_i U_d +B U_d = F
\end{equation}
where $U_d$, $A_0$, and $A_i$ are given as follows:
\[
U_d\equiv\left(\begin{array}{c} \sigma^\delta_d \\ (u^\delta_d)^t\\
\theta^\delta_d
\end{array}
\right),\; A_0 \equiv\left(\begin{array}{ccc}
\tfrac{T^\delta}{\rho^\delta}
& 0 & 0   \\ 0 & \rho^\delta \mathbb{I} & 0\\
  0 & 0 & \tfrac{3\rho^\delta}{2T^\delta}
\end{array}
\right), \; A_i \equiv\left(\begin{array}{ccc}
\tfrac{T^\delta}{\rho^\delta}(\mathfrak{u}^\delta)^i
& T^\delta e_i & 0 \\
T^\delta (e_i)^t  & \rho^\delta(\mathfrak{u}^\delta)^i \mathbb{I}&
\rho^\delta (e_i)^t\\
0 & \rho^\delta e_i &
\tfrac{3\rho^\delta}{2T^\delta}(\mathfrak{u}^\delta)^i
\end{array}
\right).
\]
$(\cdot)^t$ denotes the transpose of row vectors, $e_i$'s for
$i=1,2,3$ are the standard unit (row) base vectors in
$\mathbf{R}^3$, and $\mathbb{I}$ is the $3\times 3$ identity matrix.
$B$ and $F$, which consist of $\rho^\delta,\,
\mathfrak{u}^\delta,\,T^\delta$, $\sigma,\,u,\,\theta$ and first
derivatives of $\sigma,\,u,\,\theta$,  can be easily written down.
Note that since $\rho^\delta$ and $T^\delta$ have positive lower and
upper bounds for $t\leq \tau$, \eqref{symmetric} is strictly
hyperbolic and thus we can apply the standard energy method of the
linear symmetric hyperbolic system to \eqref{symmetric} to obtain
the following energy inequality:
\begin{equation}
\frac{d}{dt}\|U_d\|^2_{H^s}\leq C_3\|U_d \|^2_{H^s}+ C_4
\|U_d\|_{H^s}.
\end{equation}
 Here $C_3, C_4$ are constants depending on
$\|(\rho^\delta,\,\mathfrak{u}^\delta,\, T^\delta)\|_{H^{s+1}}$ and
$\|(\sigma\,,u\,,\theta)\|_{H^{s+1}}$. The second term in the right
hand side comes from the forcing term $F$. By Gronwall inequality,
we conclude that $\|(\sigma^\delta_d\,,u^\delta_d\,,
\theta^\delta_d)\|_{H^s}$ is bounded by a constant depending on
$\tau$ and $H^{s+1}$-norm of initial data
$(\sigma^0,\,u^0,\,\theta^0)$ and this completes the proof of Lemma
\ref{keylemma}.
\end{proof}


\subsection{Local Maxwellians $\protect\mu^\protect\delta$ and Proof of
Theorem \protect\ref{main-theorem}}

Now, from the refined estimate \eqref{refine-estimate}, we can choose $%
\delta_2$ sufficiently small so that for each $0<\delta\leq \delta_2$, $%
T^\delta$ satisfies the following moderate temperature variation condition
\begin{equation}  \label{T}
T_M^\delta <T^\delta(t,x) <2T_M^\delta
\end{equation}
for some constant $T_M^\delta>0$. We define
\begin{equation}  \label{delta0}
\delta_0\equiv \min\{\delta_1, \delta_2\}\,.
\end{equation}

We denote the local Mawellian, induced by compressible Euler solutions $\rho
^{\delta },\,\mathfrak{u}^{\delta },\,T^{\delta }$ with the initial data %
\eqref{pert-initial} as obtained in Lemma \ref{Life-Span}, by
\begin{equation}
\mu ^{\delta }\equiv \mu ^{\delta }(t,x,v)=\frac{\rho ^{\delta }(t,x)}{[2\pi
T^{\delta }(t,x)]^{3/2}}\exp \left\{ -\frac{[v-\mathfrak{u}^{\delta
}(t,x)]^{2}}{2T^{\delta }(t,x)}\right\} \,.  \label{LM}
\end{equation}%
For each $\delta <\delta _{0}$, we take the Hilbert expansion of the
Boltzmann equation \eqref{boltzmann} around the local Maxwellian $\mu
^{\delta }$ of the form
\begin{equation*}
F^{\varepsilon }=\sum_{n=0}^{6}\varepsilon ^{n}F_{n}+\varepsilon
^{3}F_{R}^{\varepsilon },
\end{equation*}%
where $F_{0},...,F_{6}$ are the first 6 terms of the Hilbert expansion. Here
we have set $F_{0}=\mu ^{\delta }$. Since $(\rho ^{\delta },\,\mathfrak{u}%
^{\delta },\,T^{\delta })$ is a smooth solution to the compressible Euler
system satisfying the condition \eqref{T}, from Theorem \ref{Main-Theorem}
on the compressible Euler limit, it follows that for sufficiently small $%
\varepsilon \leq \varepsilon _{0}$,
\begin{equation}
\sup_{0\leq t\leq \tau }\Vert {F^{\varepsilon }(t)-\mu ^{\delta }(t)}\Vert
_{\infty }+\sup_{0\leq t\leq \tau }\Vert {F^{\varepsilon }(t)-\mu ^{\delta
}(t)}\Vert _{2}\leq C_{\tau }\varepsilon   \label{euler-limit}
\end{equation}%
where a constant $C_{\tau }$ depends on $\tau $, $\mu ^{\delta }$, $%
F_{1},..,F_{6}$.

We now show that $\mu^\delta$ stays close to $\mu^0+\delta G$ where $G$ is
the acoustic perturbation defined in \eqref{gg}.

\begin{lemma}
\label{lemma4.1} Consider smooth solutions $(\rho^\delta,\,\mathfrak{u}%
^\delta,\, T^\delta)$ and $(\sigma\,,u\,,\theta)$ so that $%
(\sigma^\delta_d\,, u^\delta_d\,,\theta^\delta_d)$ in \eqref{perturbation}
is smooth, for instance choose $s\geq 3$ in Lemma \ref{keylemma}. Let $%
M^\delta$ be given as in \eqref{LM} and $G$ as in \eqref{gg}. Then there
exists a small enough $\delta_0>0$ so that for each $0<\delta\leq\delta_0$,
there exists a constant $C_5$ depending on $\tau$ and the initial data $%
(\sigma^0,\,u^0,\,\theta^0)$ such that
\begin{equation}  \label{expand}
\sup_{0\leq t\leq\tau}\|\mu^\delta(t)-\mu^0-\delta G (t)\|_\infty
+\sup_{0\leq t\leq\tau}\|\mu^\delta(t)-\mu^0-\delta G(t) \|_2\leq C_5\,
\delta^2.
\end{equation}
\end{lemma}

\begin{proof}  We consider
$\mu^\delta$ as a function of $\delta$ and expand it around
$\delta=0$. Instead of directly expanding $\mu^\delta$, we first
introduce auxiliary local Maxwellians $$\mu(z)\equiv
\mu^{\delta,z}=\frac{\rho^{\delta,z} (t,x)}{[2\pi
T^{\delta,z}(t,x)]^{3/2}}
\exp \left\{ -\frac{%
[v-\mathfrak{u}^{\delta,z}(t,x)]^{2}}{2T^{\delta,z}(t,x)}\right\}$$
induced by the following
$\rho^{\delta,z},\,\mathfrak{u}^{\delta,z},\,T^{\delta,z}$:
\[
\rho^{\delta,z}\equiv 1+z\sigma+z^2\sigma^\delta_d\,,\quad
\mathfrak{u}^{\delta,z}\equiv z u+z^2 u^\delta_d\,,\quad
 T^{\delta,z} \equiv 1+z\theta+z^2\theta^\delta_d\,.
\]
Compare with $\rho^\delta,\;\mathfrak{u}^\delta,\;T^\delta$ in
\eqref{perturbation}. Fix $\delta>0$. Note that $\mu(z)$ is a smooth
function of $z$, and moreover, $\mu(\delta)=\mu^\delta$, since
$\rho^{\delta,\delta}=\rho^\delta,\,\mathfrak{u}^{\delta,\delta}=
\mathfrak{u}^\delta,\,T^{\delta,\delta}=T^\delta$. Now we expand
$\mu(z)$ as a function of $z$. By Taylor's formula, $\mu(z)$ can be
written as
\begin{equation}\label{M(z)}
\mu(z)=\mu(0)+\mu'(0) z +\frac{\mu''(z_\ast)}{2}z^2
\end{equation}
for some $0\leq z_\ast\leq z$ which may depend on $(t,x,v)$ and
$\delta$. Note that $\mu(0)=\mu^0$. Denote $\frac{\partial}{\partial
z}$ by $'$. $\mu'(z)$ is given by
\[
\begin{split}
\mu'(z)&= \left\{\frac{(\rho^{\delta,z})'}{\rho^{\delta,z}}
-\frac{3(T^{\delta,z})'}{2T^{\delta,z}}
+(v-\mathfrak{u}^{\delta,z})\cdot
\frac{(\mathfrak{u}^{\delta,z})'}{T^{\delta,z}}
+\frac{|v-\mathfrak{u}^{\delta,z}|^2(T^{\delta,z})'}
{2(T^{\delta,z})^2}\right\} \mu^{\delta,z}\\
&\equiv D^{\delta,z} \mu^{\delta,z},
\end{split}
\]
where
\[
(\rho^{\delta,z})'= \sigma+2z\sigma^\delta_d\,,\quad
(\mathfrak{u}^{\delta,z})'= u+2z u^\delta_d\,,\quad
 (T^{\delta,z})' = \theta+2z\theta^\delta_d\,.
\]
Since $[(\rho^{\delta,z})', \,(\mathfrak{u}^{\delta,z})',\,
(T^{\delta,z})'](0)=[\sigma\,,u\,,\theta]$, we obtain
$$\mu'(0)=\{\sigma+ v\cdot u+(\frac{|v|^2-3}{2})\theta\}\mu^0=G(t,x,v). $$
Take one more derivative to get
\[
\mu''(z)=(D^{\delta,z})' \mu^{\delta,z}+(D^{\delta,z})^2
\mu^{\delta,z}
\]
where
\[
\begin{split}
(D^{\delta,z})'=&\frac{(\rho^{\delta,z})''}{\rho^{\delta,z}}
-\frac{((\rho^{\delta,z})')^2}{(\rho^{\delta,z})^2}
-\frac{3(T^{\delta,z})''}{2T^{\delta,z}}
+\frac{3((T^{\delta,z})')^2}{2(T^{\delta,z})^2}
-\frac{|(\mathfrak{u}^{\delta,z})'|^2}{T^{\delta,z}}\\
&+(v-\mathfrak{u}^{\delta,z})\cdot\{
\frac{(\mathfrak{u}^{\delta,z})''}{T^{\delta,z}}-
2\frac{(T^{\delta,z})'(\mathfrak{u}^{\delta,z})'}
{(T^{\delta,z})^2}\}
+|v-\mathfrak{u}^{\delta,z}|^2\{\frac{(T^{\delta,z})''}
{2(T^{\delta,z})^2}-\frac{((T^{\delta,z})')^2}{(T^{\delta,z})^3}
\}\,.
\end{split}
\]
And
$(\rho^{\delta,z})'',\,(\mathfrak{u}^{\delta,z})'',\,(T^{\delta,z})''$
are given by
\[
(\rho^{\delta,z})''= 2\sigma^\delta_d\,,\quad
(\mathfrak{u}^{\delta,z})''= 2 u^\delta_d\,,\quad
 (T^{\delta,z})'' =2\theta^\delta_d\,.
\]
Now take $z=\delta$ in \eqref{M(z)} to obtain
\[
\mu^\delta=\mu^0+ g
\delta+\frac{\mu''(\delta_\ast)}{2}\delta^2,\text{ for some
}0\leq\delta_\ast\leq\delta\,.
\]
In order to prove \eqref{expand}, it now suffices to show that
$\|\mu''(\delta_\ast)\|_\infty+\|\mu''(\delta_\ast)\|_2$ is
uniformly bounded. This follows the uniform estimates of
$\sigma_d^\delta\,,u_d^\delta\,,\theta_d^\delta$ in Lemma
\ref{keylemma}: For each $0\leq z=\delta_\ast\leq\delta\leq\delta_0$
and for $t\leq\tau$, we have the uniform point-wise estimates of
$\rho^{\delta,z},\mathfrak{u}^{\delta,z},T^{\delta,z}$,
$(\rho^{\delta,z})',(\mathfrak{u}^{\delta,z})',(T^{\delta,z})'$,
$(\rho^{\delta,z})'',(\mathfrak{u}^{\delta,z})'',(T^{\delta,z})''$
and moreover,  $\rho^{\delta,z}$ and $T^{\delta,z}$ have uniform
lower bounds for sufficiently small $\delta\leq\delta_0$. This
completes the proof of the lemma.
\end{proof}

Now \eqref{AL} is an easy consequence of Lemma \ref{lemma4.1} within the
compressible Euler limit regime.

\begin{proof}\textit{of Theorem \ref{main-theorem}:}
From \eqref{g} and \eqref{gg}, we first get
\[
\begin{split}
&\;\sup_{0\leq t\leq\tau}\|G^\varepsilon(t)- G(t)\|_\infty
+\sup_{0\leq t\leq\tau}\|G^\varepsilon(t)-G(t)\|_2\\
&=\sup_{0\leq t\leq\tau}\|\frac{F^\varepsilon(t)-\mu^0}{\delta }-
G(t)\|_\infty +\sup_{0\leq
t\leq\tau}\|\frac{F^\varepsilon(t)-\mu^0}{\delta }-G(t)\|_2\,.
\end{split}
\]
Rewrite $\frac{F^\varepsilon(t)-\mu^0}{\delta }- G(t)$ as follows:
\[
\frac{F^\varepsilon(t)-\mu^0}{\delta }-
G(t)=\frac{F^\varepsilon(t)-\mu^\delta(t)}{\delta }+
\frac{\mu^\delta(t)-\mu^0-\delta G(t)}{\delta }\,.
\]
Therefore, we obtain
\[
\begin{split}
&\sup_{0\leq t\leq\tau}\|G^\varepsilon(t)- G(t)\|_\infty
+\sup_{0\leq t\leq\tau}\|G^\varepsilon(t)-G(t)\|_2\\
&\leq \sup_{0\leq
t\leq\tau}\|\frac{F^\varepsilon(t)-\mu^\delta(t)}{\delta }\|_\infty
+\sup_{0\leq t\leq\tau}\|\frac{F^\varepsilon(t)-\mu^\delta(t)}{\delta }\|_2 \\
&\;\;+ \sup_{0\leq t\leq\tau}\|\frac{\mu^\delta(t)-\mu^0-\delta
G(t)}{\delta }\|_\infty +\sup_{0\leq
t\leq\tau}\|\frac{\mu^\delta(t)-\mu^0-\delta G(t)}{\delta }\|_2\,.
\end{split}
\]
By \eqref{euler-limit} and \eqref{expand}, the conclusion follows.
\end{proof}

\end{document}